\documentclass[12pt,a4paper]{amsart}

\usepackage{amsmath}
\usepackage{amssymb}
\usepackage{amsthm}
\usepackage{latexsym}

\usepackage{pictexwd,dcpic}

\usepackage{amscd}
\usepackage{graphicx}

\theoremstyle{plain}
\newtheorem{theorem}{Theorem}[section]
\newtheorem{corollary}[theorem]{Corollary}
\newtheorem{lemma}[theorem]{Lemma}

\theoremstyle{definition}
\newtheorem{definition}{Definition}[section]

\theoremstyle{remark}
\newtheorem{remark}{Remark}[section]
\newtheorem{example}{Example}[section]

\numberwithin{equation}{section}

\numberwithin{table}{section}

\numberwithin{figure}{section}

\newcommand{\supp}{\rm supp}
\newcommand{\NN}{\rm NN}
\newcommand{\NC}{\rm NC}

\title[Noncrossing and nonnesting partitions of types A and B]
{A bijection between noncrossing and nonnesting partitions of types
A and B}

\author{Ricardo Mamede}
\address{Ricardo Mamede, CMUC, Department of Mathematics, University of Coimbra, 3001-454
Coimbra, Portugal} \email{mamede@mat.uc.pt}

\keywords{Root systems, noncrossing partitions, nonnesting partitions, bijection}

\subjclass[2000]{05A10; 05A15}

\thanks{This work was supported by CMUC - Centro de Matem\'atica da
Universidade de Coimbra.}

\begin{document}
\maketitle
\begin{abstract}
  The total number of noncrossing  partitions of type $\Psi$ is the $n$th Catalan number $\frac{1}{n+1}\binom{2n}{n}$ when $\Psi=A_{n-1}$,
    and the binomial coefficient $\binom{2n}{n}$ when $\Psi=B_n$, and these numbers coincide with the correspondent number of nonnesting partitions. For type A, there are several
    bijective proofs of this equality; in particular, the intuitive map that locally converts each crossing to a nesting one of them.
    In this paper we present a bijection between nonnesting and noncrossing partitions of types A and B that generalizes the type A bijection that locally converts each crossing to a nesting.
\end{abstract}

\section{Introduction}

The poset of noncrossing partitions can be defined in a uniform way for any finite Coxeter
group $W$.  More precisely, for $u,w\in W$, let $u\leq w$ if there is a
shortest factorization of $w$ as a product of reflections in $W$ having as prefix  such
a shortest factorization for $u$. This partial order turns $W$ into a graded poset $Abs(W)$
having the identity $1$ as its unique minimal element, where the rank of $w$ is the length
of the shortest factorization of $w$ into reflections. Let $c$ be a Coxeter element of $W$.
Since all Coxeter elements in $W$ are conjugate to each other, the interval $[1,c]$ in $Abs(W)$
is independent, up to isomorphism, of the choice of $c$. We denote this interval by
$\NC(W)$ or by $\NC(\Psi)$, where $\Psi$ is the Cartan-Killing type of $W$, and call it the {\it poset of noncrossing partitions} of $W$. It is a self-dual, graded lattice which reduces to the classical lattice of noncrossing partitions
of the set $[n]=\{1, 2,\ldots, n\}$ defined by Kreweras in \cite{k} when $W$ is the symmetric group $\mathfrak{S}_n$ (the Coxeter group of type $A_{n-1}$), and to its type
$B$ analogue defined by Reiner in \cite{reiner} when W is the hyperoctahedral group.
The elements in $\NC(W)$ are counted by the generalized Catalan numbers,
\begin{displaymath}Cat(W)=\prod_{i=1}^k\frac{d_i+h}{d_i},\end{displaymath}
where $k$ is the number of simple reflections in $W$, $h$ is the Coxeter number and $d_1,\ldots,d_k$ are the degrees of the fundamental invariants of $W$ (see \cite{arm,grove,hum,reiner} for details on the theory of Coxeter groups and  noncrossing partitions). When $W$ is the symmetric group $\mathfrak{S}_n$, the number $Cat(\mathfrak{S}_n)$ is just the usual $n$th Catalan number $\frac{1}{n+1}\binom{2n}{n}$, and in type $B_n$ this number is the binomial coefficient $\binom{2n}{n}$.

Nonnesting partitions were defined by Postnikov (see \cite[Remark
2]{reiner}) in a uniform way for all irreducible root systems
associated with Weyl groups. If $\Phi$ is such a system, $\Phi^+$ is
a choice of positive roots, and $\Delta$ is the simple system in
$\Phi^+$, define the {\it root order}  on $\Phi^+$ by
$\alpha\leq\beta$ if $\alpha,\beta\in\Phi^+$ and $\beta-\alpha$ is
in the positive integer span of the simple roots in $\Delta$.
Equipped with this partial order,  $(\Phi^+,\leq)$ is  the {\it root
poset} of the associated Weyl group $W$. A {\it nonnesting
partition} on $\Phi$ is just an antichain in root poset
$(\Phi^+,\leq)$. Denote by $\NN(W)$ or by $\NN(\Psi)$, where $\Psi$ is
the Cartan-Killing type of $W$, the set of all nonnesting partitions
of $W$. Postnikov  showed that the nonnesting
partitions in $\NN(W)$ are also counted by the generalized Catalan
number $Cat(W)$.

 In the case of the root systems of type $A$, different bijective proofs of the equality between the cardinals  $|\NN(A_{n-1})|=|\NC(A_{n-1})|$ are known (see \cite{arm,atha,atha1,zeng,stump}), and  recently,  Christian Stump \cite{stump} described a bijection between nonnesting
and noncrossing partitions for type $B$.
Our contribution in this paper is to present a uniform proof that $|\NN(\Psi)|=|\NC(\Psi)|$, for $\Psi=A_{n-1}$ and $\Psi=B_n$,  that generalizes the
bijection presented by Armstrong in \cite{arm}.

\section{Noncrossing and nonnesting partitions of types $A$ and $B$}
A partition of the set $[n]$ is a collection of nonempty disjoint subsets of $[n]$, called blocks, whose
union is $[n]$. The type of a partition $\pi$ of $[n]$ is the integer partition formed by the cardinals of the blocks of $\pi$. Let $B$ be a block of $\pi$. Then, the least element of $B$ is called an opener, the greatest element of $B$ is said to be a closer, and the remaining elements of $B$ are called transients. The sets of openers, closers and transients of $\pi$ will be denoted by $op(\pi)$, $cl(\pi)$, and $tr(\pi)$, respectively. The triples $\mathfrak{T}(\pi)=(op(\pi),tr(\pi),cl(\pi))$ encodes useful information about the partition $\pi$. For instance, the number of blocks is $|op(\pi)|=|cl(\pi)|$, and the number of blocks having only one element is $|op(\pi)\cap cl(\pi)|$.
A partition can be graphically
represented by placing the integers $1,2,\ldots,n$ along a line and
drawing arcs above the line between $i$ and $j$ whenever $i$ and $j$
lie in the same block and no other element between them does so.

A noncrossing partition of the set $[n]$ is a partition of $[n]$ such
that there are no $a < b < c < d$, with $a,c$ belonging to some
block of the partition and $b,d$ belonging to some other block. The
set of noncrossing partitions of $[n]$, denoted by \NC$(n)$, is a
lattice for the refinement order. A nonnesting partition of the set
$[n]$ is a partition of $[n]$ such that if $a < b < c < d$ and $a,d$
are consecutive elements of a block, then $b$ and $c$ are not both
contained in some other block.  The
set of nonnesting partitions of $[n]$ will be denoted by \NN$(n)$.
Graphically, the noncrossing condition means that no two of the
arcs cross, while the nonnesting condition means that no two arcs
are nested one within the other. For instance, the noncrossing
partition $\{\{2,3\},\{1,4,5\}\}$ and the nonnesting partition $\{\{1,3\},\{2,4,5\}\}$ are represented by
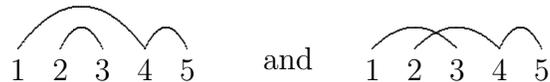
\begin{displaymath}
\begin{matrix}
\begindc{\commdiag}[8]
\obj(-4,0)[1]{$1$}[\south]\obj(-2,0)[2]{$2$}\obj(0,0)[3]{$3$}\obj(2,0)[4]{$4$}\obj(4,0)[5]{$5$}

\cmor((-2,1)(-1,2)(0,1))
\pdown(0,1){}[2]

\cmor((-4,1)(-1,3)(2,1))
\pdown(1,2){}[2]
\cmor((2,1)(3,2)(4,1))
\pdown(-2,2){}[2]
\enddc&\quad\text{ and }\quad
\begindc{\commdiag}[8]
\obj(-4,0)[1]{$1$}[\south]\obj(-2,0)[2]{$2$}\obj(0,0)[3]{$3$}\obj(2,0)[4]{$4$}\obj(4,0)[5]{$5$}

\cmor((-4,1)(-2,2)(0,1)) \pdown(-2,2){}[2]

\cmor((-2,1)(0,2)(2,1)) \pdown(-2,2){}[2] \cmor((2,1)(3,2)(4,1))
\pdown(-2,3){}[2]
\enddc
\end{matrix}
\end{displaymath}
respectively. Both partitions have $\{1,2\}$ as set of openers, $\{3,5\}$ as set of closers and $4$ is the only transient.
As pointed out in \cite{arm}, the map that locally converts
each crossing to a nesting
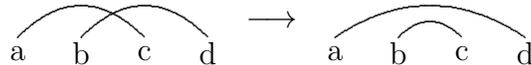
\begin{displaymath}\begindc{\commdiag}[6]
\obj(-16,0)[1]{a}[\south]\obj(-12,0)[2]{b}\obj(-8,0)[3]{c}\obj(-4,0)[4]{d}\obj(0,2)[5]{$\longrightarrow$}

\cmor((-16,1)(-12,3)(-8,1)) \pdown(0,3){}[2]
\cmor((-12,1)(-8,3)(-4,1)) \pdown(0,3){}[2]

\obj(4,0)[1]{a}[\south]\obj(8,0)[2]{b}\obj(12,0)[3]{c}\obj(16,0)[4]{d}
\cmor((4,1)(10,3)(16,1)) \pdown(0,3){}[2] \cmor((8,1)(10,2)(12,1))
\pdown(0,3){}[2]
\enddc\end{displaymath}
defines a bijection from \NN$(n)$ to \NC$(n)$ that preserves the
number of blocks. We will refer to this bijection as the L-map.

We will now review the usual combinatorial realizations of the
Coxeter groups of types $A$ and $B$, referring to \cite{hum} for any undefined terminology.
The Coxeter group  $W$ of type  $A_{n-1}$ is realized combinatorially as the symmetric
group $\mathfrak{S}_n$. The permutations in $\mathfrak{S}_n$ will be written in cycle notation. The simple generators of $\mathfrak{S}_n$ are  the transpositions of adjacent integers $(i\; i+1)$, for $i=1,\ldots,n-1$, and the reflections are the transpositions $(i\; j)$ for $1 \leq i < j \leq n $. To any permutation
$\pi\in \mathfrak{S}_n$ we associate the partition of the set
$[n]$ given by its cycle structure. This defines a isomorphism
between the posets \NC$(\mathfrak{S}_n)$ of noncrossing partitions of $\mathfrak{S}_n$, defined in the introduction, and \NC$(n)$, with respect to
the Coxeter element $c=(12\cdots n)$ \cite[Theorem 1]{biane}.

Denoting by $e_1,\ldots,e_n$ the standard basis of $\mathbb{R}^n$, the
root system of type $A_{n-1}$ consists of the set of vectors
\begin{displaymath}\Phi=\{e_i-e_j:i\neq j,1\leq i,j\leq n\},\end{displaymath}
each root $e_i-e_j$ corresponding to the transposition $(i\;j)$.
Take
\begin{displaymath}\Phi^+=\{e_i-e_j\in\Phi:i>j\}\end{displaymath} for the set of
positive roots and, defining $r_i:=e_{i+1}-e_{i}$, $i=1,\ldots,n-1$,
we obtain the simple system $\Delta=\{r_1,\ldots,r_{n-1}\}$ for
$\mathfrak{S}_n$. Note that
\begin{displaymath}e_i-e_j=\sum_{k=j}^{i-1}r_k,\qquad\text{ if }i>j\,.\end{displaymath}

The correspondence between the antichains in the root poset
$(\Phi^+,\leq)$ and the set of nonnesting partitions of $[n]$ is
given by the bijection which sends the positive root $e_i-e_j$ to the set partition of $[n]$ having  $i$ and $j$ in the same block. For instance, consider the root poset
$(\Phi^+,\leq)$ of type $A_4$:
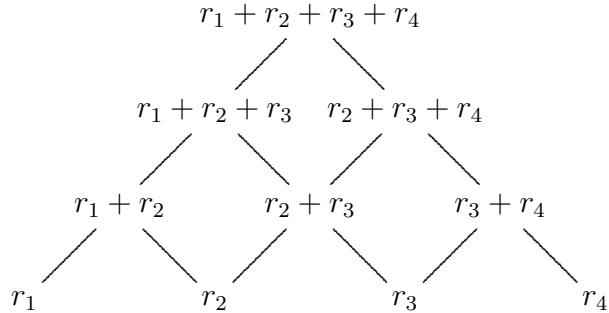
\begin{displaymath}
\begindc{\commdiag}[9]
\obj(-12,0)[1]{$r_1$}\obj(-4,0)[2]{$r_2$}\obj(4,0)[3]{$r_3$}
 \obj(12,0)[4]{$r_4$}

 \obj(-8,4)[12]{$r_1+r_2$} \obj(0,4)[23]{$r_2+r_3$} \obj(8,4)[34]{$r_3+r_4$}

 \obj(-4,8)[123]{$r_1+r_2+r_3$} \obj(4,8)[234]{$r_2+r_3+r_4$}

 \obj(0,12)[1234]{$r_1+r_2+r_3+r_4$}

 \mor{1}{12}{}[\atright,\solidline]\mor{2}{12}{}[\atright,\solidline]
 \mor{2}{23}{}[\atright,\solidline]\mor{3}{23}{}[\atright,\solidline]
 \mor{3}{34}{}[\atright,\solidline]\mor{4}{34}{}[\atright,\solidline]

 \mor{12}{123}{}[\atright,\solidline]\mor{23}{123}{}[\atright,\solidline]
 \mor{23}{234}{}[\atright,\solidline]\mor{34}{234}{}[\atright,\solidline]

 \mor{123}{1234}{}[\atright,\solidline]\mor{234}{1234}{}[\atright,\solidline]
\enddc
\end{displaymath}
The antichain $r_1+r_2=e_3-e_1$ corresponds to the transposition
$(13)$ in the symmetric group $\mathfrak{S}_5$, and thus to the
nonnesting set partition $\{\{1,3\},\{2\},\{4\},\{5\}\}$, while the
antichain $\{r_1+r_2,r_2+r_3,r_4\}$ corresponds to the product of
transpositions $(13)(24)(45)=(13)(245)$ in $\mathfrak{S}_5$, and
thus to the nonnesting set partition $\{\{1,3\},\{2,4,5\}\}$.

Given a positive root $\alpha=r_i+r_{i+1}+\cdots+r_j\in\Phi^+$,
define the {\it support} of $\alpha$ as the set
\supp$(\alpha)=\{r_i,r_{i+1},\ldots,r_j\}$. The integers $i$ and $j$ will be called, respectively, the {\it first} and {\it last indices}
of $\alpha$, and the roots $r_i$ and $r_j$ the first and last elements of $\alpha$, respectively.  We have the following lemma.

\begin{lemma}\label{l1}
    Let $\alpha_1,\alpha_2$ be two roots in $\Phi^+$ with first and last indices ${i_1},{j_1}$ and ${i_2},{j_2}$, respectively. Then,  $\alpha_1,\alpha_2$ form an antichain if and only
if $i_1<i_2$ and $j_1<j_2$. 
\end{lemma}

\medskip

Consider now the Coxeter group  $W$ of type  $B_n$, with its usual combinatorial realization
as the hyperoctahedral group of signed permutations of \begin{displaymath}[\pm n]:=\{\pm 1,\pm 2,\ldots,\pm n\}.\end{displaymath} These
are permutations of $[\pm n]$ which commute with
the involution $i\mapsto -i$. We will write the elements of $W$ in cycle notation, using commas between elements.
The simple generators of $W$ are the transposition $(-1,\; 1)$ and the pairs $(-i-1,\; -i)(i,\; i+1)$ for
$i=1,\ldots,n-1$. The reflections in $W$ are the transpositions $(-i,\; i)$, for $i=1,\ldots,n$, and the pairs of
transpositions $(i,\; j)(-j,\; -i)$ for $i\neq j$.
Identifying the sets $[\pm n]$ and $[2n]$ through the map  $i\mapsto i$ for $i\in[n]$ and $i\mapsto n-i$ for $i\in\{-1,-2,\ldots,-n\}$,
allows us to identify the hyperoctahedral group $W$ with the subgroup $U$ of $\mathfrak{S}_{2n}$
 which commutes with  the permutation  $(1,n+1)(2,n+2)\cdots (n,2n)$. For example, the signed permutations $(1,3)$ and $(2,-3)(-2,3)$ in the
 hyperoctahedral group of type $B_3$ correspond to the permutations $(1\;3)$ and $(2\;6)(5\;3)$ in the symmetric group $\mathfrak{S}_6$.   It follows that \NC$(U)$ is a sublattice of \NC$(\mathfrak{S}_{2n})$, isomorphic to \NC$(W)$ (see \cite{arm}).

The type $B_n$ root system consists on the set of $2n^2$ vectors
\begin{displaymath}\Phi=\{\pm e_i:1\leq i\leq n\}\cup\{\pm e_i\pm e_j:i\neq j, 1\leq i,j\leq n\},\end{displaymath}
and we take
\begin{displaymath}\Phi^+= \{e_i:1\leq i\leq n\}\cup\{e_i\pm e_j:1\leq j<i\leq n\}\end{displaymath}
as a choice of positive roots.  Changing
the notation slightly from the one used for $\mathfrak{S}_n$, let
$r_1:=e_1$ and $r_i:=e_i-e_{i-1}$, for $i=2,\ldots,n$. The set
\begin{displaymath}\Delta:=\{r_1,r_2,\ldots,r_n\}\end{displaymath}
is a simple system for $W$, and easy computations show that
\begin{align*}
    e_i=\sum_{k=1}^ir_k,&\\
    e_i-e_j=\sum_{k=j+1}^{i}r_{k},&\qquad\text{ if }i>j\\
    e_i+e_j=2\sum_{k=1}^{j}r_k+\sum_{k=j+1}^{i}r_{k},&\qquad\text{ if }i>j\,.
\end{align*}
Each root $e_i,e_i-e_j$ and
$e_i+e_j$ defines a reflection that acts on  $\mathbb{R}^n$ as the permutation $(i,\;-i)$,
$(i,\;j)(-i,\;-j)$ and $(i,\;-j)(-i,\;j)$, respectively, and  we will identify the roots with the corresponding permutations.
For example, consider the root poset of type $B_3$  displayed below:
 
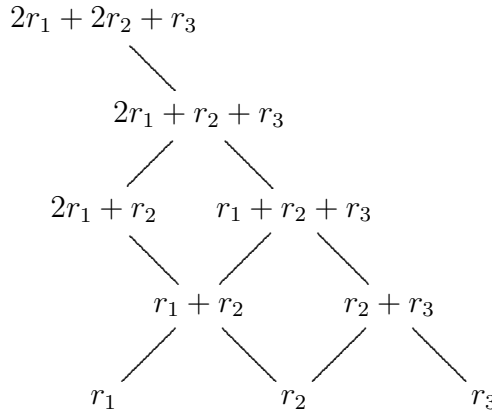
\begin{displaymath}
\begindc{\commdiag}[9]
\obj(-4,0)[2]{$r_1$}\obj(4,0)[3]{$r_2$}
 \obj(12,0)[4]{$r_3$}

  \obj(0,4)[23]{$r_1+r_2$} \obj(8,4)[34]{$r_2+r_3$}

 \obj(-4,8)[123]{$2r_1+r_2$} \obj(4,8)[234]{$r_1+r_2+r_3$}

 \obj(0,12)[1234]{$2r_1+r_2+r_3$}\obj(-4,16)[11234]{$2r_1+2r_2+r_3$}

 \mor{2}{23}{}[\atright,\solidline]\mor{3}{23}{}[\atright,\solidline]
 \mor{3}{34}{}[\atright,\solidline]\mor{4}{34}{}[\atright,\solidline]

 \mor{23}{123}{}[\atright,\solidline]
 \mor{23}{234}{}[\atright,\solidline]\mor{34}{234}{}[\atright,\solidline]

 \mor{123}{1234}{}[\atright,\solidline]\mor{234}{1234}{}[\atright,\solidline]
 \mor{1234}{11234}{}[\atright,\solidline]
\enddc
\end{displaymath}
The antichain $\{2r_1+r_2,r_2+r_3\}$ corresponds to the signed
permutation $(1,-2,-3)(-1,2,3)$.

Using the inclusion $W\hookrightarrow\mathfrak{S}_{2n}$  specified
above, we may represent noncrossing and nonnesting
partitions of $W$ graphically using the conventions made for its type $A$
analogs. In these representations, we use the integers
$$-1,-2,\ldots,-n,1,2,\ldots,n,\quad\text{ or }\quad -n,\ldots,-2,-1,0,1,2,\ldots,n,$$
 respectively for noncrossing
and nonnesting partitions, instead of the usual $1,2,\ldots,2n$, where the presence of the zero in the ground set for nonnesting partitions is necessary to correctly represent (when present)
the arc between a positive number $i$ an its negative (see \cite{atha}).

Given  a noncrossing or a nonnesting partition $\pi$ of type $B_n$, let the set of openers $op(\pi)$ be formed by the least element of all blocks of $\pi$ having only positive integers; let the set of closers $cl(\pi)$ be formed by the greatest element of all blocks of $\pi$ having only positive integers and by the absolute values of the least and greatest elements of all blocks having positive and negative integers; and finally let the set of transients $tr(\pi)$ be formed by all elements of $[n]$ which are not in $op(\pi)\cup cl(\pi)$. For instance,  if $\pi$ is the nonnesting partition $\{\{-4,4\},\{-1,2\},\{-2,1\},\{3,5\},\{-3,-5\}\}$, represented below,
then $op(\pi)=\{3\}$, $cl(\pi)=\{1,2,4,5\}$ and $tr(\pi)=\emptyset$.
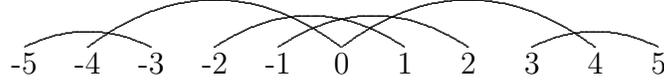
\begin{displaymath}\begindc{\commdiag}[6]
\obj(0,0)[1]{0}[\south]\obj(4,0)[1]{1}[\south]\obj(8,0)[1]{2}[\south]\obj(12,0)[1]{3}[\south]\obj(16,0)[1]{4}[\south]
\obj(20,0)[1]{5}[\south]\obj(-4,0)[1]{-1}[\south]\obj(-8,0)[1]{-2}[\south]\obj(-12,0)[1]{-3}[\south]\obj(-16,0)[1]{-4}[\south]
\obj(-20,0)[1]{-5}[\south]

\cmor((12,1)(16,2)(20,1)) \pdown(0,3){}[2]
\cmor((-12,1)(-16,2)(-20,1)) \pdown(0,3){}[2]
\cmor((-8,1)(-2,3)(4,1)) \pdown(0,3){}[2]
\cmor((8,1)(2,3)(-4,1)) \pdown(0,3){}[2]
\cmor((-16,1)(-8,4)(0,1)) \pdown(0,4){}[2]
\cmor((16,1)(8,4)(0,1)) \pdown(0,4){}[2]
\enddc\end{displaymath}

A factor $2r_i$ appearing in a positive root $\alpha$ will be called
a {\it double root} of $\alpha$. The {\it support} of a positive
root $\alpha\in\Phi^+$ is the set of simple and double roots in
$\alpha$. Define also the set $\overline{\supp}(\alpha)$ as the set formed by the simple roots appearing in alpha as simple or double roots. For instance, for
$\alpha=2r_1+\cdots+2r_j+r_{j+1}+\cdots+r_{i}$ and
$\beta=r_{\ell}+\cdots+r_k$ we have
\supp$(\alpha)=\{2r_1,\ldots,2r_j,r_{j+1},\ldots,r_i\}$, $\overline{\supp}(\alpha)=\{r_1,\ldots,r_i\}$, and
$\supp(\beta)=\overline{\supp}(\beta)=\{r_{\ell},\ldots,r_k\}$. The first and last indices
of $\alpha$ are, respectively $1,i$ and $\ell,k$. The integer $j$ will be called the last double index in $\alpha$. Define also the set $D_{\alpha}:=\{r_2,\ldots,r_j\}$ as the set of simple roots appearing in $\alpha$ as double roots, other than $r_1$.
We have the following lemma.

\begin{lemma}\label{l2}
    Let $\alpha$ and $\beta$ be two roots in $\Phi^+$ with first and last indices $i,j$ and $i',j'$, respectively.
     If neither $\alpha$ nor $\beta$ have double roots, then $\{\alpha,\beta\}$ is an antichain if and only if $i<i'$ and $j<j'$. If $\alpha$ has double roots, then $\{\alpha,\beta\}$ is an antichain if and only if  $j<j'$ and the number of double roots in $\alpha$ is greater than the
     number of double roots in $\beta$.
\end{lemma}

\section{Main result}

Let $\Phi$ denote a root system of type $A$ or type $B$, and let $\Phi^+$ and $\Delta$ be defined as above. In view of lemmas \ref{l1} and \ref{l2}, we consider antichains $\{\alpha_1,\ldots,\alpha_m\}$ in $\Phi^+$  as ordered $m$-tuples numbered so that if ${i_{\ell}}$ is the last index of $\alpha_{\ell}$, then $i_1<\cdots<i_m$.

\begin{definition}
Given two positive roots $\alpha$ and $\beta$, with $\beta$ having no double roots, and such that the intersection of their supports is nonempty, define their union $\alpha\cup\beta$ and their intersection $\alpha\cap\beta$ as the positive roots with supports
$$\supp(\alpha\cup\beta):=\supp(\alpha)\cup\left(\supp(\beta)\setminus\overline{\supp}(\alpha)\right)
\quad\text{and}\quad\supp(\alpha\cap\beta):=\overline{\supp}(\alpha)\cap\supp(\beta),$$
respectively.  If moreover $\alpha$ has double roots, then define also their d-intersection $\alpha\cap^d\beta$ as the positive root with support $\supp(\alpha\cap^d\beta):=D_{\alpha}\cap\supp(\beta)$.
\end{definition}

\begin{example} The union and the intersections of the type $B_3$
positive roots $\alpha=2r_1+2r_2+r_3$ and $\beta=r_2+r_3+r_4$ are
$\alpha\cup\beta=2r_1+2r_2+r_3+r_4$, $\alpha\cap^d\beta=r_2$, and $\alpha\cap\beta=r_2+r_3$.
\end{example}

An antichain $(\alpha_1,\ldots,\alpha_m)$ is said to be {\it connected} if the intersection of the supports of any two adjacent roots $\alpha_i,\alpha_{i+1}$ is non empty.
The connected components \begin{displaymath}(\alpha_1,\ldots,\alpha_i), (\alpha_{i+1},\ldots,\alpha_j),\ldots,(\alpha_k,\ldots,\alpha_m)\end{displaymath} of an antichain $\alpha=(\alpha_1,\ldots,\alpha_m)$ are the connected sub-antichains of $\alpha$ for which the supports of the union of the roots in any two distinct components are disjoint.
For instance,  the
antichain $(r_1+r_2,r_2+r_3,r_4)$ has the connected components $(r_1+r_2,r_2+r_3)$ and $r_4$. We will use lower and upper arcs to match two roots in a connected antichain in a geometric manner. Two roots linked by a lower [respectively upper] arc are said to be l-linked [respectively, u-linked]. In what follows we will identify each root with the correspondent permutation.

\begin{definition}\label{def1}
Define the map $f$ from the set $\NN(\Phi)$ into $\NC(\Phi)$ recursively as follows.
When $\alpha_1$ is a positive root we set $f(\alpha_1):=\alpha_1$.
If
$\alpha=(\alpha_1,\ldots,\alpha_m)$ is a  connected antichain with $m> 2$, we have two cases:

$(a)$ If there are no double roots in the antichain, define \begin{displaymath}f(\alpha):=\left(\underset{k=1}{\overset{m}{\bigcup}}\alpha_k\right)f(\overline{\alpha}_2,\ldots,\overline{\alpha}_m),\end{displaymath}
where $\overline{\alpha}_k=\alpha_{k-1}\cap\alpha_k$ for $k=2,\ldots,m$.

\medskip

$(b)$ Assume now that $\alpha_1,\ldots,\alpha_{\ell}$ have double
roots, for some $\ell\geq 1$, and $\alpha_{\ell+1},\ldots,\alpha_m$
have none. Let
$\Gamma_d:=(\alpha_1,\ldots,\alpha_{\ell})$ and $\Gamma:=(\alpha_{\ell+1},\ldots,\alpha_m)$.
We start by introducing l-links as follows.

Let $m'$ be the largest index of elements in $\Gamma$ such that the following holds: $\alpha_{m'}$ has a first index $i\neq 1$, so that there is a rightmost element, say $\alpha_{k}$, of $\,\Gamma_d$ which has a term $2r_i$. If there is such an integer $m'$, l-link $\alpha_{k}$ with $\alpha_{m'}$.
Then, ignore $\alpha_{k}$ and $\alpha_{m'}$ and proceed with the remaining roots as before. This procedure terminates after a finite number of steps (and not all elements of $\alpha$ need to be l-linked).

Next proceed by introducing u-links in $\alpha$.
The starting point of u-links, which we consider drawn from right to left, will be elements in $\Gamma$ that have no first index $1$ and are not l-linked. We will refer to these elements as admissible roots. So, let $m'$ be the smallest integer such that the following holds: $\alpha_{m'}$ is an admissible root with first index $i\neq 1$ so that there is a leftmost element, say $\alpha_k$ which has $r_i$ or $2r_i$ in its support and is not yet u-linked to an element on its right. If there is such an integer $m'$, u-link $\alpha_{k}$ with $\alpha_{m'}$. Remove $\alpha_{m'}$ from the set of admissible roots and proceed as before. Again this process terminates after a finite number of steps.

Finally,  let $T=\{t_1<\cdots<t_p\}$ be the collection of all last double indices of the roots in $\Gamma_d$ not l-linked, and all the last indices of the roots in $\alpha$ not u-linked to an element on its right.
 Then, define
\begin{displaymath}f(\alpha):=\pi_1\cdots\pi_{\ell}\pi_0\,\theta_1\,\cdots\theta_qf(\theta_{q+1},\ldots,\theta_s),\end{displaymath}
where for $j=1,\ldots,\ell$,
$\pi_j=2r_1+\cdots+2r_{j'}+r_{j'+1}+\cdots+r_{j''}$,
with $j'$ and $j''$ respectively the leftmost and rightmost integers in $T$ not considered yet; $\pi_0$ is either the root
$r_1+\cdots+r_{i_j}$,  if the first index of  $\alpha_{\ell+1}$ is
$1$, with $i_j$  the only integer in $T$ not used yet for
defining the roots $\pi_j$, or the identity
otherwise; each $\theta_j$, $j=1,\ldots,q$ is the d-intersection of l-linked roots, starting from the rightmost one in $\Gamma_d$, and each $\theta_j$, $j=q+1,\ldots,s$ is the intersection of u-linked roots, starting from the leftmost one in $\Gamma$.

\medskip

$(c)$ For the general case, if  $(\alpha_1,\ldots,\alpha_i), (\alpha_{i+1},\ldots,\alpha_j),\ldots,(\alpha_k,\ldots,\alpha_m)$ are the connected components of $(\alpha_1,\ldots,\alpha_m)$, let
\begin{displaymath}f(\alpha_1,\ldots,\alpha_m):=f(\alpha_1,\ldots,\alpha_i)f(\alpha_{i+1},\ldots,\alpha_j)\cdots f(\alpha_k,\ldots,\alpha_m).\end{displaymath}
\end{definition}

\begin{remark}
$(i)$ Notice that in the type $A$ case, the map $f$ is defined only by conditions $(a)$ and $(c)$ of the above definition.
Also, note that if all roots in $\alpha$ have double roots then condition $(b)$ is vacuous and  the map $f$ reduces to the identity map. We point out that the number of roots in $f(\alpha)$ is equal to the number of roots in the antichain $\alpha$.

$(ii)$ The sequence  $(\overline{\alpha}_2,\ldots,\overline{\alpha}_m)$ obtained in step $(a)$ is a (not necessarily connected) antichain.
    It is easy to check that after all l-links and all u-links are settled, the set $T$ has an odd number of elements if and only if the first index of $\alpha_{\ell+1}$ is $1$. Thus, the root $\pi_0$ given in condition $(b)$ is well defined.

\end{remark}

We will show that $f$ establishes a bijection between the sets $\NN(\Psi)$ and $\NC(\Psi)$, for $\Psi=A_{n-1}$ or $\Psi=B_n$. Before, however, we present some examples.

\begin{example}\label{exnn1}
    Consider the antichain $\alpha=(r_1+r_2,r_2+r_3,r_3+r_4+r_5,r_4+r_5+r_6,r_5+r_6+r_7)$ in the root poset of type $A_7$, corresponding to the permutation $(136)(247)(58)$ in the symmetric group $\mathfrak{S}_8$. Applying the map $f$ to $\alpha$, we get the noncrossing partition
    \begin{align*}
    f(\alpha)&=(r_1+\cdots+r_7)f(r_2,r_3,r_4+r_5,r_5+r_6)\\
    &=(r_1+\cdots+r_7)r_2r_3f(r_4+r_5,r_5+r_6)\\
    &=(r_1+\cdots+r_7)r_2r_3(r_4+r_5+r_6)r_5\\
    &\equiv(18)(2347)(56),
    \end{align*}
whose graphical representation is given below:
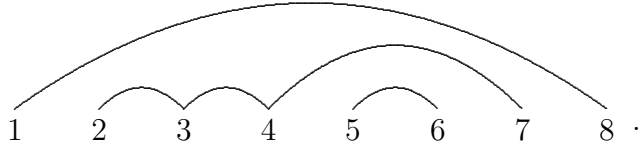
\begin{displaymath}\begindc{\commdiag}[8]
\obj(-14,0)[1]{$1$}\obj(-10,0)[2]{$2$}\obj(-6,0)[3]{$3$}\obj(-2,0)[4]{$4$}
\obj(2,0)[5]{$5$}\obj(6,0)[6]{$6$}\obj(10,0)[7]{$7$}\obj(14,0)[8]{$8$}
\cmor((-14,1)(0,6)(14,1))
\pdown(0,6){}[2]
\cmor((-10,1)(-8,2)(-6,1))
\pdown(2,3){}[2]
\cmor((-6,1)(-4,2)(-2,1))
\pdown(2,3){}[2]
\cmor((-2,1)(4,4)(10,1))
\pdown(2,3){}[2]
\cmor((2,1)(4,2)(6,1))
\pdown(2,3){}[2]
\enddc\;\;.\end{displaymath}

\end{example}

\begin{example}\label{exnn2}
    Consider now the antichain $\alpha=(\alpha_1,\alpha_2,\alpha_3,\alpha_4,\alpha_5)$ in the root poset $B_9$, where
    \begin{align*}
    \alpha_1&=2r_1+2r_2+2r_3+2r_4+r_5,&
    \alpha_2&=2r_1+2r_2+r_3+r_4+r_5+r_6,\\
    \alpha_3&=r_1+r_2+r_3+r_4+r_5+r_6+r_7,
    &\alpha_4&=r_3+r_4+r_5+r_6+r_7+r_8,\\
    \alpha_5&=r_4+r_5+r_6+r_7+r_8+r_9.
    \end{align*}

    Following definition \ref{def1}, we get the l-links and the u-links shown below:
    \begin{displaymath}\begindc{\commdiag}[10]
\obj(-14,0)[1]{$\alpha=(\alpha_1,\quad\alpha_2,\quad\alpha_3,\quad\alpha_4,\quad\alpha_5)$}
\cmor((-18,1)(-15,2)(-11,1))
\pdown(0,2){}[2]
\cmor((-18,-1)(-13,-2)(-8,-1))
\pdown(0,-2){}[2]
\enddc.\end{displaymath}
  Therefore,
    $T=\{2,6,7,8,9\}$ and the application of $f$ to $\alpha$
    yields:
    \begin{align*}
    f(\alpha)&=(2r_1+2r_2+r_3+\cdots+r_9)(2r_1+\cdots+2r_6+r_7+r_8)(r_1+\cdots+r_7)r_4
    f(r_3+r_4+r_5)\\
    &\equiv(2,-9)(-2,9)(6,-8)(-6,8)(7,-7)(3,4)(-3,-4)(2,5)(-2,-5)\\
    &=(2,5,-9)(-2,-5,9)(6,-8)(-6,8)(7,-7)(3,4)(-3,-4).
    \end{align*}
    The image $f(\alpha)$ is a noncrossing partition in $[\pm 9]$, as we may check in its representation:
    
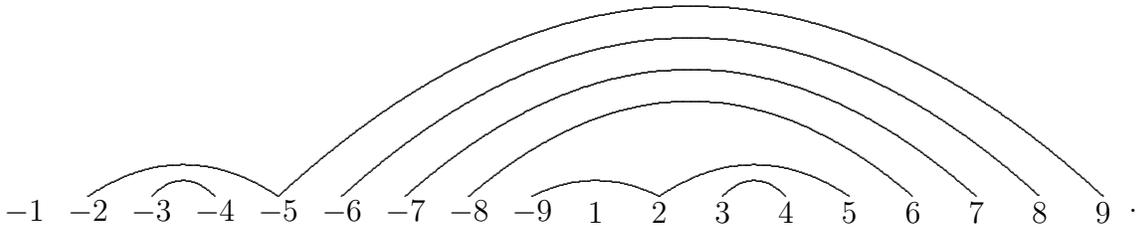
\begin{displaymath}\begindc{\commdiag}[6]
\obj(-22,0)[4]{$-4$}\obj(-18,0)[5]{$-5$}\obj(-14,0)[6]{$-6$}\obj(-10,0)[7]{$-7$}\obj(-6,0)[8]{$-8$}\obj(-2,0)[9]{$-9$}
\obj(-34,0)[1]{$-1$}\obj(-30,0)[2]{$-2$}\obj(-26,0)[3]{$-3$}
\obj(2,0)[-1]{$1$}\obj(6,0)[-2]{$2$}\obj(10,0)[-3]{$3$}\obj(14,0)[-4]{$4$}\obj(18,0)[-5]{$5$}\obj(22,0)[-6]{$6$}
\obj(26,0)[-7]{$7$}\obj(30,0)[-8]{$8$}\obj(34,0)[-9]{$9$}

\cmor((-30,1)(-24,3)(-18,1))
\pdown(13,13){}[2]
\cmor((34,1)(8,13)(-18,1))
\pdown(13,13){}[2]
\cmor((-22,1)(-24,2)(-26,1))
\pdown(13,13){}[2]
\cmor((-14,1)(8,11)(30,1))
\pdown(13,13){}[2]
\cmor((-10,1)(8,9)(26,1))
\pdown(13,13){}[2]
\cmor((-6,1)(8,7)(22,1))
\pdown(13,13){}[2]
\cmor((6,1)(12,3)(18,1))
\pdown(13,13){}[2]
\cmor((-2,1)(2,2)(6,1))
\pdown(13,13){}[2]
\cmor((10,1)(12,2)(14,1))
\pdown(13,13){}[2]

\enddc\;\;.\end{displaymath}
\end{example}

\medskip

\begin{example}
    For a final example, consider the antichain $\alpha=(\alpha_1,\ldots,\alpha_6)$ in the root poset of type $B_{11}$, where
    \begin{align*}
    \alpha_1&=2r_1+2r_2+2r_3+2r_4+2r_5+r_6,&\alpha_2&=2r_1+2r_2+2r_3+2r_4+r_5+r_6+r_7,\\
    \alpha_3&=2r_1+2r_2+2r_3+r_4+r_5+r_6+r_7+r_8,&\alpha_4&=r_2+r_3+r_4+r_5+r_6+r_7+r_8+r_9,\\
    \alpha_5&=r_5+r_6+r_7+r_8+r_9+r_{10},\text{ and }&\alpha_6&=r_7+r_8+r_9+r_{10}+r_{11}.\\
    \end{align*}
    The l-links and u-links are shown below, so that $T=\{4,6,8,9,10,11\}$:
    \begin{displaymath}\begindc{\commdiag}[10]
\obj(-14,0)[1]{$\alpha=(\alpha_1,\quad\alpha_2\;,\quad\alpha_3\;,\quad\alpha_4\;,\quad\alpha_5\;,\quad\alpha_6)$}
\cmor((-14,-1)(-13,-2)(-12,-1))
\pdown(0,-2){}[2]
\cmor((-20,-1)(-14,-3)(-9,-1))
\pdown(0,-3){}[2]
\cmor((-17,1)(-12,2)(-7,1))
\pdown(0,2){}[2]
\enddc.\end{displaymath}

 Therefore, the application of $f$ to $\alpha$ gives
    \begin{math}f(\alpha)=\pi_1\pi_2\pi_3(r_2+r_3)(r_5)f(r_7),\end{math}
    where
    \begin{align*}
    \pi_1&=2r_1+2r_2+2r_3+2r_4+r_5+r_6+r_7+r_8+r_9+r_{10}+r_{11},\\
    \pi_2&=2r_1+2r_2+2r_3+2r_4+2r_5+2r_6+r_7+r_8+r_9+r_{10},\text{ and }\\ \pi_3&=2r_1+2r_2+2r_3+2r_4+2r_5+2r_6+2r_7+2r_8+r_9.
    \end{align*}
     Thus, $f(\alpha)$ is the noncrossing partition

    \begin{align*}
    &(4,-11)(-4,11)(6,-10)(-6,10)(8,-9)(-8,9)(1,3)(-1,-3)(4,5)(-4,-5)(6,7)(-6,-7)\\
    &=(4,5,-11)(-4,-5,11)(6,7,-10)(-6,-7,10)(8,-9)(-8,9)(1,3)(-1,-3)
    \end{align*}
     represented below:
    \begin{displaymath}\begindc{\commdiag}[5]
\obj(-22,0)[6]{$-6$}\obj(-18,0)[7]{$-7$}\obj(-14,0)[8]{$-8$}\obj(-10,0)[9]{$-9$}\obj(-6,0)[10]{$-10$}\obj(-2,0)[11]{$-11$}
\obj(-34,0)[3]{$-3$}\obj(-30,0)[4]{$-4$}\obj(-26,0)[5]{$-5$}
\obj(-42,0)[1]{$-1$}\obj(-38,0)[2]{$-2$}

\obj(2,0)[-1]{$1$}\obj(6,0)[-2]{$2$}\obj(10,0)[-3]{$3$}\obj(14,0)[-4]{$4$}\obj(18,0)[-5]{$5$}\obj(22,0)[-6]{$6$}
\obj(26,0)[-7]{$7$}\obj(30,0)[-8]{$8$}\obj(34,0)[-9]{$9$}\obj(38,0)[-10]{$10$}\obj(42,0)[-11]{$11$}

\cmor((-42,1)(-38,3)(-34,1))
\pdown(13,13){}[2]
\cmor((2,1)(6,3)(10,1))
\pdown(13,13){}[2]

\cmor((-30,1)(-28,2)(-26,1))
\pdown(13,13){}[2]
\cmor((-26,1)(9,18)(42,1))
\pdown(18,20){}[2]
\cmor((14,1)(16,2)(18,1))
\pdown(18,20){}[2]
\cmor((-2,1)(6,5)(14,1))
\pdown(18,20){}[2]

\cmor((-14,1)(10,11)(34,1))
\pdown(18,20){}[2]
\cmor((-10,1)(10,9)(30,1))
\pdown(18,20){}[2]

\cmor((-22,1)(-20,2)(-18,1))
\pdown(18,20){}[2]
\cmor((-18,1)(10,15)(38,1))
\pdown(18,20){}[2]

\cmor((22,1)(24,2)(26,1))
\pdown(18,20){}[2]
\cmor((-6,1)(8,7)(22,1))
\pdown(18,20){}[2]

\enddc\;\;.\end{displaymath}
\end{example}

\bigskip

\begin{lemma}\label{lemaB}
    If $\alpha\in$ \NN$(B_n)$ then $f(\alpha)\in$ \NC$(B_n)$, and $\mathfrak{T}(\alpha)=\mathfrak{T}(f(\alpha))$.
\end{lemma}
\begin{proof}
 Let $\alpha=(\alpha_1,\ldots,\alpha_m)$ be
an antichain in
    the root poset of type $B_n$, and let $(\alpha_1,\ldots,\alpha_w)$ be its first connected component.
Start by assuming that none of the positive roots in $\alpha$ have the simple root $r_1$ nor the double root $2r_1$. We will use induction on $m\geq 1$ to show that in this case $f(\alpha)$ is a noncrossing partition on the set $\{i-1,\ldots,q,-(i-1),\ldots,-q\}$, where $i$ is the fist index of $\alpha_1$ and $q$ is the last index of $\alpha_m$, and such that each positive integer is sent to another positive integer.
The result is clear when $m=1$. So, let $m\geq 2$ and assume the result for antichains of length less than, or equal to $m-1$. Then, we may write
\begin{displaymath}f(\alpha)=\left(\underset{k=1}{\overset{w}{\bigcup}}\alpha_k\right)f(\overline{\alpha}_2,\ldots,\overline{\alpha}_w)
f(\alpha_{w+1},\ldots,\alpha_m),\end{displaymath}
where each $\overline{\alpha}_k=\alpha_{k-1}\cap\alpha_k$, for $k=2,\ldots,w$. By the inductive step,
$f(\overline{\alpha}_2,\ldots,\overline{\alpha}_w)\equiv\pi_1$ and
$f(\alpha_{w+1},\ldots,\alpha_m)\equiv\pi_2$ are noncrossing partitions on the sets \begin{displaymath}\{a-1,\ldots,b,-(a-1),\ldots,-b\}\,\text{ and }\,\{p-1,\ldots,q,-(p-1),\ldots,-q\},\end{displaymath} respectively, where
$a$ and $p$ are the first indices of $\overline{\alpha}_2$ and $\alpha_{w+1}$, respectively, and $b$ and $q$ are the last indices of $\overline{\alpha}_w$ and $\alpha_m$, respectively. Moreover, all positive integers are sent to positive ones by $\pi_1$ and $\pi_2$.
Denoting by $j$ the last index of $\alpha_w$, we get
$\underset{k=1}{\overset{w}{\bigcup}}\alpha_k=r_i+\cdots+r_j\equiv(i-1,j)(-(i-1),-j)$ with
$i-1<a-1<b<j\leq p-1<q.$
Therefore \begin{displaymath}f(\alpha)\equiv(i-1,j)(-(i-1),-j)\pi_1\pi_2\end{displaymath} is a noncrossing partition on the set $\{i-1,\ldots,q,-(i-1),\ldots,-q\}$ sending each positive integer to another positive integer.

Note that for the rest of the proof, we may assume without
loss of generality that
 $\alpha$ is connected, since none of the  connected components of an antichain,
except possible for the first one, have double roots, and therefore
their images are noncrossing partitions sending each positive
integer to another positive integer.

\medskip

Suppose now that the first element of $\alpha_1$ is $r_1$. We will show that $f(\alpha)$ is a noncrossing partition on the set $\{i-1,\ldots,q,-(i-1),\ldots,-q\}$, where $i$ is the first index of $\alpha_2$ and $q$ is the last index of $\alpha_m$, and such that one and only one positive integer is sent to a negative one.
The result is certainly true for $m=1$, and when $m>1$ we have
\begin{displaymath}f(\alpha)=\left(\underset{k=1}{\overset{m}{\bigcup}}\alpha_k\right)f(\overline{\alpha}_2,\ldots,\overline{\alpha}_m),\end{displaymath}
where $\underset{k=1}{\overset{m}{\cup}}\alpha_k\equiv(q,-q)$,  and $\overline{\alpha}_k=\alpha_{k-1}\cap\alpha_k$ for $k=2,\ldots,m$. By the previous case, $f(\overline{\alpha}_2,\ldots,\overline{\alpha}_m)\equiv\pi$ is a noncrossing partition on the set $\{i-1,\ldots,j,-(i-1),\ldots,-j\}$, with $i$ the first index of $\alpha_2$ and $j<q$ the last index of $\alpha_{m-1}$.
Therefore, $f(\alpha)\equiv (q,-q)\pi$ is a noncrossing partition satisfying the desired conditions.


\medskip

Next, assume that $\alpha$ satisfies condition $(b)$ of definition \ref{def1}, and consider its image
\begin{displaymath}f(\alpha)=\pi_1\cdots\pi_{\ell}\pi_0\,\theta_1\,\cdots\theta_qf(\theta_{q+1},\ldots,\theta_s).\end{displaymath}
By the construction of the set $T$, it follows that each $D_{\alpha_j}$, $j=1,\ldots,\ell$, is contained in $D_{\pi_i}$, for some $i=1,\ldots,\ell$, and that $\pi_1\cdots\pi_{\ell}\pi_0$ is a noncrossing partition, sending each nonfixed positive integer to a negative one.
Note also that the support of each $\theta_j$, $j=1,\ldots,q$, is contained in some $D_{\alpha_i}$, $i=1,\ldots,\ell$, and therefore, in some $D_{\pi_i}$, $i=1,\ldots,\ell$. Moreover, the supports of any two roots $\theta_i$ and $\theta_j$, $1\leq i,j\leq q$, are either disjoint, or one of them is contained into the other one. Therefore $\theta_1\cdots\theta_q$ is a noncrossing partition sending each nonfixed positive integer into another positive integer.
By the previous cases, $f(\theta_{q+1},\ldots,\theta_s)$ is also a noncrossing partition sending each nonfixed positive integer into another positive integer. Again by the construction of the set $T$, we find that the support of each $\theta_j$, $j=q+1,\ldots,s$, is either contained in some $D_{\pi_i}$, or it does not intersect $D_{\pi_{\ell}}$. For each $j=1,\ldots,q$ and $i=q+1,\ldots,s$, either we have $\supp(\theta_i)\cap\supp(\theta_j)=\emptyset$, or $\supp(\theta_i)\supseteq\supp(\theta_j)$, this last case happening when $\theta_i$ arises from the intersection of two u-linked roots $\alpha_u\in\Gamma_d$ and $\alpha_v\in\Gamma$, and there is some $\alpha_{v+k}\in\Gamma$, $k\geq 1$, l-linked to $\alpha_u$, whose d-intersection gives $\theta_j$. Therefore, it follows that $f(\alpha)$ is noncrossing.

To see that $\mathfrak{T}(\alpha)=\mathfrak{T}(f(\alpha))$, notice that if $a\in op(\alpha)$, then $\alpha$ must have a root $r_{a+1}+\cdots+r_b$, for some $b>a+1$, and cannot have neither a root with  last double index equal to $a$ nor a root with last index equal to $a$. By its construction, the same is true for the sequence $f(\alpha)$, and therefore $a\in op(f(\alpha))$.
Assume now that $a\in tr(\alpha)$. Then, $a$ must appear either as the last index or last double index of a root and $a+1$ as the first index of another root. Again the same will happen in the sequence $f(\alpha)$, and thus $a\in tr(f(\alpha))$. It follows that $\mathfrak{T}(\alpha)=\mathfrak{T}(f(\alpha))$.
 \end{proof}

    With some minor adaptations, the proof of lemma \ref{lemaB}, in the case where neither the simple root $r_1$ nor the double root $2r_1$ are present in $\alpha$, gives the  type $A$ analog of the previous result.

     \begin{corollary}
    If $\alpha\in$ \NN$(A_{n-1})$ then $f(\alpha)\in$ \NC$(A_{n-1})$, and $\mathfrak{T}(\alpha)=\mathfrak{T}(f(\alpha))$.
\end{corollary}

We will now construct the inverse function of $f$, thus showing that $f$ establishes a bijection between the sets $\NN(\Psi)$ and $\NC(\Psi)$, for  $\Psi=A_{n-1}$ or $\Psi=B_n$. For that propose, recall the following  property.

\begin{lemma}\label{trans}
    Two distinct transpositions $(a,b)$ and $(i,j)$ in $\mathfrak{S}_n$ commute if and only if the sets $\{i,j\}$ and $\{a,b\}$ are disjoint.
\end{lemma}

If $\pi_1\cdots\pi_p$ is the cycle structure of a signed permutation $\pi$, then for each cycle $\pi_i=(ij\cdots k)$ there is another cycle $\pi_j=(-i-j\cdots-k)$. Denote by $\pi'_i$ the cycle in $\{\pi_i,\pi_j\}$ having the smallest
positive integer (when $\pi_i=\pi_j$ then $\pi'_i$ is just $\pi_i$), and call positive cycle structure to the subword of $\pi_1\cdots\pi_p$
formed by the cycles $\pi'_i$. Extend this definition to permutations in $\mathfrak{S}_n$ by identifying positive cycle structure with cycle structure.

\begin{theorem}\label{def2}
    The map $f$ is a bijection between the sets $\NN(\Psi)$ and $\NC(\Psi)$, for $\Psi=A_{n-1}$ or $\Psi=B_n$, which preserves the triples $(op(\pi),cl(\pi),tr(\pi))$.
\end{theorem}
\begin{proof}
We will construct the inverse map $g:\NC(\Psi)\rightarrow\NN(\Psi)$
of $f$. Given $\pi\in\NC(\Psi)$, let
$\pi_1\cdots\pi_s$ be its positive cycle structure. Replace each cycle $\pi_i=(i_1i_2\cdots i_k)$  by $(i_1i_2)(i_2i_3)\cdots(i_{k-1}i_k)$, if $i_{\ell}>0$ for $\ell=1,\ldots,k$, or by
$$\pi_i=(i_1i_{j+1})(i_1i_2)(i_2i_3)\cdots(i_{j-1}i_j)(i_{j+1}i_{j+2})\cdots(i_{k-1}i_{k}),$$ if
$i_{\ell}>0$ for $\ell=1,\ldots,j$, and $i_{\ell}<0$ for $\ell=j+1,\ldots,k$. Next,
baring in mind lemma \ref{trans} and recalling that $\pi$ is noncrossing, move all transpositions $(i,j)$, with $i>0$ and $j<0$ (if any), to the leftmost
   positions and order them by its least positive element, and order all remaining transpositions $(i,j)$, with $i,j>0$, by its least positive
   integer. Replace each transposition $(ij)$ by its correspondent root in the root system of type $\Psi$, and let
\begin{equation}\label{seqnc}
(\alpha_1,\ldots,\alpha_k)(\alpha_{k+1},\ldots,\alpha_{\ell})\cdots(\alpha_m,\ldots,\alpha_n)
\end{equation}
be the correspondent sequence of roots, divided by its connected
components.  Note that given two distinct roots in \eqref{seqnc}, the sets formed by the first and last indices, if there are no double roots, or by the last and last double indices, otherwise, are clearly disjoint.

We start by considering  that the sequence \eqref{seqnc} has only one connected component $(\alpha_1,\ldots,\alpha_k)$.
Let $\Gamma_d=(\alpha_1,\ldots,\alpha_r)$ be the subsequence formed by the roots having double roots, and denote by $\Gamma=(\alpha_{r+1},\ldots,\alpha_k)$ the remaining subsequence. Define $\Gamma'=\Gamma'_d=\emptyset$. If $\Gamma_d$ is not empty and $r\neq k$, apply the following algorithm:

 Let $\overline{\Gamma}$ be the subsequence of $\Gamma$ obtained by striking out the root $\alpha_{r+1}$ if its first index is 1.
While $\overline{\Gamma}\neq\emptyset$, repeat the following steps:

 $(i)$ Let $\alpha_i$ be the leftmost root in $\overline{\Gamma}$ and
check if $\supp(\alpha_i)\subseteq D_{\alpha_j}$, for some $\alpha_j\in\Gamma_d\setminus\Gamma'_d$.

$(ii)$ If so, let $\alpha_{i_j}$ be the rightmost root in $\Gamma_d\setminus\Gamma'_d$ with this property. Update $\Gamma'$ by including in it the rightmost root $\overline{\alpha}$ of  $\overline{\Gamma}$ whose support is contained in $\supp(\alpha_i)$. Update $\overline{\Gamma}$ by striking out the root $\overline{\alpha}$ and update $\Gamma'_d$ by including in this set the root $\alpha_{i_j}$.

 $(iii)$ Otherwise, update $\overline{\Gamma}$ by striking out the root $\alpha_i$.

 Next, let $T=\{t_1>\cdots>t_r\}$ be the set formed by all last double indices of the roots in $\Gamma_d\setminus\Gamma'_d$ and by the last indices of the roots in $\Gamma'$;
 let $F_{st}=\{f_{r+1}<\cdots<f_k\}$ be the set formed by the first indices of the roots in $\Gamma$, and let $L_{st}=\{\ell_1<\cdots<\ell_k\}$ be the set formed by the last indices  of the roots in $(\Gamma\setminus\Gamma')\cup\Gamma_d$ and by the last double indices of the roots in $\Gamma'_d$. By this construction, we have $f_i<\ell_i$ for $i=1,\ldots,r$, and $f_i<\ell_i$, for $i=r+1,\ldots,k$. Then, define
 $$g(\pi)=(\overline{\alpha}_1,\ldots,\overline{\alpha}_k),$$
 where for $i=1,\ldots,r$, $\overline{\alpha}_i=2r_1+\cdots+2r_{t_i}+r_{t_i+1}+\cdots+r_{\ell_i}$,
 and for $i=r+1,\ldots,k$, $\overline{\alpha}_i=r_{f_i}+\cdots+r_{\ell_i}$.

\medskip

For the general case define
$$g(\pi)=g(\alpha_1,\ldots,\alpha_k)g(\alpha_{k+1},\ldots,\alpha_{\ell})\cdots g(\alpha_m,\ldots,\alpha_n).$$

It is clear from this construction that $g(\pi)$ is an antichain in the root poset of type $\Psi$. Moreover, a closer look at the construction of the map $f$ shows that $g$ is the inverse of $f$. Thus,  $f$ (and $g$)
establishes a bijection between nonnesting and noncrossing
partitions of types $A$ and $B$.
\end{proof}

In the following examples we illustrate the application of the map
$g$.

\begin{example}
    Consider the cycle structure of the noncrossing partition $\pi=(18)(2347)(56)$ in the symmetric group $\mathfrak{S}_8$ used in example \ref{exnn1}.
    Following the proof of theorem \ref{def2}, write
    \begin{align*}
        \pi&\equiv(18)(2347)(56)\\
        &=(18)(23)(34)(47)(56)\\
        &\equiv (r_1+\cdots+r_7)r_2r_3(r_4+r_5+r_6)r_5
    \end{align*}
    Note that $\pi$ has only one connected component, and there are no double roots. Next define the sets
    $$F_{st}=\{1,2,3,4,5\},\,\text{ and }L_{st}=\{2,3,5,6,7\}.$$
    Thus, we find that the image of $\pi$ by the map $g$ is the antichain
    $$g(\pi)=(r_1+r_2,r_2+r_3,r_3+r_4+r_5,r_4+r_5+r_6,r_5+r_6+r_7).$$
\end{example}

\begin{example}
Consider now the noncrossing partition $$\pi=(2,5,-9)(-2,-5,9)(6,-8)(-6,8)(7,-7)(3,4)(-3,-4)$$ obtained in example \ref{exnn2}. Its positive cycle structure is
\begin{align*}
(2,-9)(2,5)(6,-8)(7,-7)(3,4)=(2,-9)(6,-8)(7,-7)(2,5)(3,4),
\end{align*}
and thus we get $$\pi\equiv(2r_1+2r_2+r_3+\cdots+r_9,2r_1+\cdots+2r_6+r_7+r_8,r_1+\cdots+r_7,r_3+r_4+r_5,r_4).$$
Next, construct the sets
\begin{align*}
&T=\{4,2\},\quad F_{st}=\{1,3,4\}\\
&L_{st}=\{5,6,7,8,9\}.
\end{align*}
Therefore, the image of $\pi$ by the map $g$ is the antichain
$$(2r_1+\cdots+2r_4+r_5,2r_1+2r_2+r_3+\cdots+r_6,r_1+\cdots+r_7,r_3+\cdots+r_8,r_4+\cdots+r_9).$$
\end{example}

\medskip

Finally, in the next result we prove that the map $f$ generalizes
the bijection that locally converts
each crossing to a nesting.

\begin{theorem}
    When restricted to the type $A_{n-1}$ case, the map $f$ coincides with
    the L-map.
\end{theorem}
\begin{proof}
Let $\alpha=(\alpha_1,\ldots,\alpha_m)$ be an antichain in the root
poset of type $A_{n-1}$. The result will be handled by induction
over $m\geq 1$.
    Without loss of generality, we may assume that $\alpha$ is
    connected, since otherwise there is an integer $1<k<n-1$ such that
    each integer less (resp. greater) than $k$ is sent by $\alpha$ to an integer that
    still is less (resp. greater) that $k$. Therefore, the same
    happens with the image of $\alpha$  by either the map $f$ or the
    L-map.

    The result is vacuous when $m=1$, and when $m=2$, the only
    connected nonnesting partition which does not stay invariant
    under the maps $f$ and $L$ is
    $\alpha=(r_i+\cdots+r_{i'})(r_j+\cdots+r_{j'})$, for some
    integers
    $1\leq i<j<i'<j'\leq n-1$. In this case, the equality between $f$ and the
    L-map is obvious. So, let $m>2$ and assume the result for
    antichains of length $\leq m-1$. Let $i$ and $j$ be,
    respectively, the first and last indices of $\alpha_1$ and
    $\alpha_m$. Then,
    \begin{displaymath}f(\alpha)=(r_i+\cdots+r_j)f(\overline{\alpha}_2,\ldots,\overline{\alpha}_m),\end{displaymath}
    where each $\overline{\alpha}_k=\alpha_{k-1}\cap\alpha_k$ for $k\geq 2$, and the antichain $(\overline{\alpha}_2,\ldots,\overline{\alpha}_m)$ is
    clearly nonnesting, and not necessarily connected. By the
    inductive step,
    $f(\overline{\alpha}_2,\ldots,\overline{\alpha}_m)=L(\overline{\alpha}_2,\ldots,\overline{\alpha}_m)$.
    Moreover, note that converting, from left to right, each local
    crossing between the first root and the leftmost root in $\alpha$ whose arcs cross, into a nesting gives, precisely,
    \begin{displaymath}(r_i+\cdots+r_j)L(\overline{\alpha}_2,\ldots,\overline{\alpha}_m),\end{displaymath}
    and this operation may be considered the first step of the
    L-map. Thus, we find that $f(\alpha)=L(\alpha)$.
\end{proof}

\begin{example}
Consider the antichain
$\alpha=(r_1+r_2+r_3,r_2+r_3+r_4+r_5,r_3+r_4+r_5+r_6,r_5+r_6+r_7)$
in the root poset of type $A_7$. Applying the map $f$ we get
\begin{align*}
f(\alpha)&=(r_1+\cdots+r_7)f(r_2+r_3,r_3+r_4+r_5,r_5+r_6)\\
&=(r_1+\cdots+r_7)(r_2+r_3+r_4+r_5+r_6)f(r_3,r_5)\\
&=(r_1+\cdots+r_7)(r_2+r_3+r_4+r_5+r_6)r_3r_5\equiv (18)(27)(34)(56).
 \end{align*} On the other hand, applying the L-map to each
crossing between the first root and the leftmost root in $\alpha$
whose arcs cross, we get successively
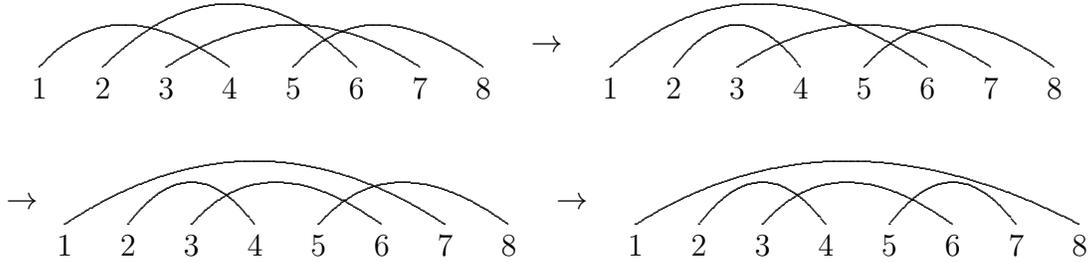
\begin{displaymath}\begindc{\commdiag}[8]
\obj(-24,0)[1]{1}\obj(-21,0)[2]{2}\obj(-18,0)[3]{3}\obj(-15,0)[4]{4}\obj(-12,0)[5]{5}\obj(0,2)[a]{$\rightarrow$}
\obj(-9,0)[6]{6}\obj(-6,0)[7]{7}\obj(-3,0)[8]{8}

\obj(24,0)[-8]{8}\obj(21,0)[-7]{7}\obj(18,0)[-6]{6}\obj(15,0)[-5]{5}\obj(12,0)[-4]{4}
\obj(9,0)[-3]{3}\obj(6,0)[-2]{2}\obj(3,0)[-1]{1}

\cmor((-24,1)(-20,3)(-15,1)) \pdown(0,4){}[2]
\cmor((-21,1)(-15,4)(-9,1)) \pdown(0,4){}[2]
\cmor((-18,1)(-12,3)(-6,1)) \pdown(0,4){}[2]
\cmor((-12,1)(-8,3)(-3,1)) \pdown(0,4){}[2]

\cmor((3,1)(10,4)(18,1)) \pdown(0,4){}[2] \cmor((6,1)(9,3)(12,1))
\pdown(0,4){}[2] \cmor((9,1)(15,3)(21,1)) \pdown(0,4){}[2]
\cmor((15,1)(19,3)(24,1)) \pdown(0,4){}[2]

\enddc\end{displaymath}

\begin{displaymath}\begindc{\commdiag}[8]
\obj(-26,2)[a]{$\rightarrow$}
\obj(-24,0)[1]{1}\obj(-21,0)[2]{2}\obj(-18,0)[3]{3}\obj(-15,0)[4]{4}\obj(-12,0)[5]{5}\obj(0,2)[a]{$\rightarrow$}
\obj(-9,0)[6]{6}\obj(-6,0)[7]{7}\obj(-3,0)[8]{8}

\obj(24,0)[-8]{8}\obj(21,0)[-7]{7}\obj(18,0)[-6]{6}\obj(15,0)[-5]{5}\obj(12,0)[-4]{4}
\obj(9,0)[-3]{3}\obj(6,0)[-2]{2}\obj(3,0)[-1]{1}

\cmor((-24,1)(-15,4)(-6,1)) \pdown(0,4){}[2]
\cmor((-21,1)(-18,3)(-15,1)) \pdown(0,4){}[2]
\cmor((-18,1)(-14,3)(-9,1)) \pdown(0,4){}[2]
\cmor((-12,1)(-8,3)(-3,1)) \pdown(0,4){}[2]

\cmor((3,1)(13,4)(24,1)) \pdown(0,4){}[2] \cmor((6,1)(9,3)(12,1))
\pdown(0,4){}[2] \cmor((9,1)(13,3)(18,1)) \pdown(0,4){}[2]
\cmor((15,1)(18,3)(21,1)) \pdown(0,4){}[2]
\enddc\end{displaymath}

Thus, in the first step of the L-map, we get
\begin{math}L(\alpha)=(r_1+\cdots+r_7)L(r_2+r_3,r_3+r_4+r_5,r_5+r_6).\end{math}
Continuing the application of the L-map, now replacing, by a
nesting, each crossing between the second root and the leftmost root
in $\alpha$ whose arcs cross, we get
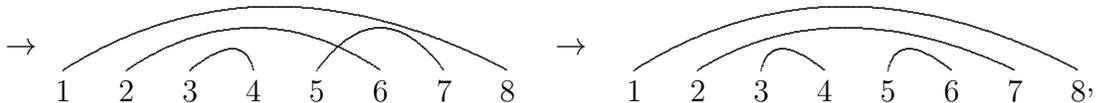
\begin{displaymath}\begindc{\commdiag}[8]
\obj(-26,2)[a]{$\rightarrow$}\obj(-24,0)[1]{1}\obj(-21,0)[2]{2}\obj(-18,0)[3]{3}\obj(-15,0)[4]{4}\obj(-12,0)[5]{5}\obj(0,2)[a]{$\rightarrow$}
\obj(-9,0)[6]{6}\obj(-6,0)[7]{7}\obj(-3,0)[8]{8}

\obj(24,0)[-8]{8}\obj(21,0)[-7]{7}\obj(18,0)[-6]{6}\obj(15,0)[-5]{5}\obj(12,0)[-4]{4}
\obj(9,0)[-3]{3}\obj(6,0)[-2]{2}\obj(3,0)[-1]{1}

\cmor((-24,1)(-14,4)(-3,1)) \pdown(0,4){}[2]
\cmor((-21,1)(-15,3)(-9,1)) \pdown(0,4){}[2]
\cmor((-18,1)(-16,2)(-15,1)) \pdown(0,4){}[2]
\cmor((-12,1)(-9,3)(-6,1)) \pdown(0,4){}[2]

\cmor((3,1)(13,4)(24,1)) \pdown(0,4){}[2] \cmor((6,1)(13,3)(21,1))
\pdown(0,4){}[2] \cmor((9,1)(10,2)(12,1)) \pdown(0,4){}[2]
\cmor((15,1)(16,2)(18,1)) \pdown(0,4){}[2]
\enddc,\end{displaymath}
and therefore, we have
\begin{math}L(\alpha)=(r_1+\cdots+r_7)(r_2+r_3+r_4+r_5+r_6)r_3r_5=f(\alpha).\end{math}
\end{example}

\emph{Acknowledgements}.
The author thanks Christian Krattenthaler and Alexander Kovacec for  useful comments and suggestions.

\end{document}